\documentclass[amsart,11pt]{article}
\usepackage{setspace} 
\usepackage{amsmath, amsthm, amssymb}
\usepackage{graphicx}
\usepackage{epsfig}

\newtheoremstyle{dotless}{}{}{\itshape}{}{\bfseries}{}{ }{}
\def\Real{\hbox{I\kern-.1667em\hbox{R}}}
\theoremstyle{dotless}	
\input epsf

\title{Square donuts and twistable holes}

\author{Kevin Murawski\\
    University of Notre Dame\\
    Notre Dame, Indiana  46556 \\
    kmurawsk@nd.edu\\
    \\
    Neil R. Nicholson\\
    University of Notre Dame\\
    Notre Dame, Indiana  46556 \\
    nnichol3@nd.edu\\
    \\
    Kathleen Walsh\\
    University of Notre Dame\\
    Notre Dame, Indiana  46556 \\
    kwalsh24@nd.edu\\
	}

\newtheoremstyle{dotless}{}{}{\itshape}{}{\bfseries}{}{ }{}
\def\Real{\hbox{I\kern-.1667em\hbox{R}}}
\theoremstyle{dotless}	
\input epsf

\newtheorem{thm}{Theorem}[section]

\begin{document}

\maketitle

\begin{abstract}
A mathematical donut is a rectangle of integral side length with a smaller rectangle (called the hole of the donut), also of integral side length, strictly inside it and with sides of the rectangles parallel to each other, where the area of the larger rectangle is twice that of the smaller.  Necessary and sufficient conditions are determined for when the hole of the donut can be rotated $90^{\circ}$ and a donut still exists, and a complete classification of all square or square-holed donuts is given, with the square donut classification being intimately related to Pythagorean triples.  \\


\end{abstract}

\section{Introduction}

In their paper, Nirode and Krumpe \cite{2} describe stumbling across a routine problem in a high school geometry book.  It describes a trend in middle age landscape design.  Rectangular courtyards regularly contained a rectangular garden, whose sides were parallel to those of the courtyard itself.  For aesthetic reasons, it was prescribed that the the garden's area would be exactly half that of the courtyard.  

Without further restrictions, this is simple to accomplish: place a divider down the middle of the courtyard and assign half the courtyard to be the garden.  However, the geometry text insisted that a path of fixed width surround the garden (so that the courtyard path surrounding the garden would be of the same area as the garden) \cite{1}.  Nirode and Krumpe did not insist on a fixed width border around the garden, but they did require the integral side lengths of both rectangles.  Thus, throughout this paper, all numbers referenced are positive integers.

Define a \textit{rectangular donut $D$} to be a number $D = ab$, where \\ $1 < b \leq a < D$ and $D = 2xy$, with $1 \leq x < a$ and $1 \leq y < b$.  Visually, $D$ is an $a \times b$ rectangle (called the \textit{exterior} of $D$) with a smaller $x \times y$ rectangle (referred to as the \textit{hole} of $D$) inside of it, where corresponding sides of the two rectangles are parallel.  In the notation of \cite{2}, $D$ will be written as the ordered quadruple $(a,b,x,y)$.

The connection proven in \cite{2} was that set of rectangular donuts is in correspondence with sums of Pythagorean triples (Theorem \ref{DonutThm} below).  But this is not a one-to-one correspondence due to the fact that a fixed donut exterior may have multiple hole variations.  For example, the donut $84$ can be realized with three different configurations: $(28,3,21,2)$, $(21,4,14,3)$, and $(12,7,7,6)$.  Using the terminology of Pythagorean triples, a donut $(a,b,x,y)$ is called \textit{primitive} if there is only a single configuration with the pairs $(a,x)$ and $(b,y)$ being relatively prime (that is, the corresponding lengths of the exterior and hole are share no common factors).  A donut $(a,b,x,y)$ is called \textit{quasi-primitive} when it has multiple configurations satisfying the requirement that corresponding pairs of side lengths are relatively prime.

Noting that the donuts less than $100$ that are primitive or quasi-primitive are $12$, $30$, $40$, $56$, $70$, $84$, and $90$, and that these are precisely the first ordered perimeters of primitive Pythagorean triangles (A024364, \cite{3}), the following theorem was proven.

\begin{thm} \label{DonutThm}
    {\rm \cite{2}} A number is a primitive or quasi-primitive rectangular donut if and only if it is the sum of a primitive Pythagorean triple.
\end{thm}

This result, coupled with the fact that multiplying a Pythagorean triple by a constant yields additional Pythagorean triples (that is, the primitive donut $(a,b,x,y)$ yields additional nonprimitive donuts $(ha,kb,hx,ky)$), provides the aforementioned correspondence between Pythagorean triples and rectangular donuts. 

The focus of this paper is geometric.  We aim to completely classify both square and square-holed donuts (that is, any donut of the form $(a,a,x,y)$ or $(a,b,x,x)$, respectively).  
 Note that square-holed square donuts do not exist.  The proof of this is left to the reader.  Square donuts are a subcollection of donuts with \textit{twistable holes}; $(a,b,x,y)$ has a twistable hole if $(a,b,y,x)$ is also a donut.  In addition to classifying square and square-holed donuts, a necessary and condition for a hole to be \textit{twistable} is given.  

\section{Twistable holes and squares}

The notation used for a donut prescribes an orientation for the hole; in $(22,20,20,11)$, the exterior side of length $22$ is parallel to the hole side of length $20$.  In this case, the hole of the donut could not be rotated $90^{\circ}$; $(22,20,11,20)$ is not a donut.  In some donuts this is possible, such as $(21,20,15,14)$ and $(21,20,14,15)$.  This is an example of the previously defined \textit{twistable} donut.  
 
To determine if $(a,b,x,y)$ has a twistable hole, one need only consider two inequalities relative to the largest dimension of the donut's exterior.

\begin{thm}
The donut $(a,b,x,y)$ has a twistable hole if and only if both $2x > a$ and $2y>a$.
\end{thm}
\begin{proof}
Given that $(a,b,x,y)$ is a donut, we know 

\begin{equation} \label{eqn0}
 ab = 2xy,   
\end{equation}

\noindent as well as 

\begin{equation} \label{eqn1}
    1 \leq x < a
\end{equation}

\noindent and 

\begin{equation} \label{eqn2}
    1 \leq y < b.
\end{equation}

If we assume its hole is twistable, we additionally have that 

\begin{equation} \label{eqn3}
    1 \leq y < a
\end{equation}

\noindent and 

\begin{equation} \label{eqn4}
    1 \leq x < b.
\end{equation}

\noindent By this last inequality, 

\begin{equation}
2xy < 2yb, 
\end{equation}

\noindent or equivalently, 

\begin{equation}
    ab < 2xb, 
\end{equation}

\noindent showing that $a < 2x$.  Simililarly, using Eqn. \ref{eqn4}, we have that $a < 2y$.

Now suppose $2x> a$ and $2y > a$.  To show that the hole is twistable, we must show that the inequalities of Eqns. \ref{eqn3} and \ref{eqn4} hold.  By Eqn. \ref{eqn2} and since $b \leq a$, Eqn. \ref{eqn3} holds.  Then, the assumption $a < 2y$ yields

\begin{equation}
    ab < 2yb,
\end{equation}

\noindent or equivalently, 

\begin{equation}
    2xy < 2yb,
\end{equation}

\noindent so that $x < b$, as required.

\end{proof}

The hole of every square or square-holed donut is clearly twistable; $(12,12,8,9)$ and $(50,36,30,30)$ are two examples, respectively, that when the holes are twisted become the donuts $(12,12,9,8)$ and $(50,36,30,30)$.  Theorems \ref{square1} and \ref{square2} provide a complete classification of all such donuts.

\begin{thm} \label{square1}
 The donut $(a,b,n,n)$ is realizable if and only if there exist relatively prime factors $p$ and $q$ of $n$ with $p < q < 2p$.  
\end{thm}
\begin{proof}
    Suppose $p|n$ and $q|n$ with gcd$(p,q) = 1$ and $p < q < 2p$.  Then, defining $a = 2pn/q$ and $b = qn/p$ yields a donut, since both 

    \begin{equation}
    1 < \dfrac{2p}{q},
\end{equation}

\noindent and

\begin{equation}
    1 < \dfrac{q}{p},
\end{equation}

\noindent so that $a > n$ and $b > n$, as well as $ab = 2n^2$.

Now, suppose that $(a,b,n,n)$ is a donut, so that $a > n$, $b > n$ and $ab = 2n^2$.  Let $n$ have prime factorization given by

\begin{equation}
    n = p_1^{e_1} \ldots p_k^{e_k}.
\end{equation}

\noindent Note that it is impossible for $k = 1$, for in this scenario, 

\begin{equation}
    ab = 2p_1^{2e^1},
\end{equation}

\noindent and the product $2p_1^{2e_1}$ cannot be partitioned into two terms so that both

\begin{equation}
    a > 2p_1^{\alpha_1}
\end{equation}

\noindent and 

\begin{equation}
    b > p_1^{\alpha_2}.
\end{equation}

Thus, $k \geq 2$.  In order to show the existence of $p$ and $q$, because

\begin{equation}
    \left( \dfrac{a}{n} \right) \left( \dfrac{b}{n} \right) = 2,
\end{equation}

\noindent it must be the case that, without loss of generality,

\begin{equation}
    \dfrac{a}{n} = \dfrac{2p}{q}
\end{equation}

\noindent and 

\begin{equation}
    \dfrac{b}{n} = \dfrac{q}{p},
\end{equation}

\noindent where $p$ and $q$ are a product of some of the prime factors of $n$.  Becase $k \geq 2$ and that we can simply reduce each fraction to lowest terms, we have that ${\rm gcd}(p,q) = 1$.  Lastly, since both $a > n$ and $b > n$, the result follows since both

\begin{equation}
    \dfrac{2p}{q} > 1
\end{equation}

\noindent and 

\begin{equation}
    \dfrac{q}{p} > 1.
\end{equation}
\end{proof}

Having classified all square-holed donuts, we move now to classifying all square donuts.  It is worth reminding ourselves of the result in \cite{2}.  A number (which is the product of the dimensions of the outer rectangle of the donut) is a donut if and only if it is the sum of a Pythagorean triple.  In Thm. \ref{square2}, the classification of donuts of the form $(n,n,a,b)$ depends not on $n^2$ (as Thm. \ref{DonutThm} does) but $n$ itself being the sum of a Pythagorean triple.

To illustrate this results, consider the donut $(12,12,9,8)$.  Because \\$12^2 = 16 + 63 + 65$, with $(16,63,65)$ indeed being a Pythagorean triple, Thm. \ref{DonutThm} guarantees it is a donut.  However, if we were searching for square donuts, because $12 = 3+4+5$, we would know one exists.  However, even though $(10,9,9,5)$ is a donut (because $90$ is the sum of a Pythagorean triple), it would be fruitless to attempt to find a square donut of side length $9$ or $10$, as these numbers themselves are not the sum of a Pythagorean triple.

\begin{thm} \label{square2}
The donut of side length $n$ (that is, $(n,n,a,b)$ for some $a$ and $b$) is realizable if and only if $n$ is the sum of a Pythagorean triple.    
\end{thm}
\begin{proof}
    Begin by assuming $n = x+y+z$ for a Pythagorean triple $(x,y,z)$.  Then, $x = k(p^2-q^2)$, $y = k(2pq)$, and $z = k(p^2 + q^2)$, via Euclid's formula for Pythagorean triples (where $q < p$).  Then, 

    \begin{equation} \label{squarehole1}
        n = k(2p^2+2pq),
    \end{equation}

    \noindent so that 

    \begin{equation} \label{squarehole2}
        n^2 = 2 \left[ 2k^2p^2(p+q)^2\right].
    \end{equation}

    If we set $a = 2kp^2$ and $b = k(p+q)^2$, we claim that $(n,n,a,b)$ is a donut.  As Eqn. \ref{squarehole2} demonstrates the area requirement to be a donut, we need only show that both $a < n$ and $b < n$.  Equation \ref{squarehole1} gives $a < n$.  Since $q < p$, we have that 

    \begin{equation}
        p^2 + 2pq + q^2 < p^2 + 2pq + p^2,
    \end{equation}

    \noindent or equivalently, 

    \begin{equation}
        (p+q)^2 < 2p^2 + 2pq,
    \end{equation}

    \noindent yielding $b < n$.

    Next, assume that $(n,n,a,b)$ is a donut.  Because $n^2 = 2ab$, the donut is not primitive (since gcd$(n,a) \neq 1$ and gcd$(n,b) \neq 1$).  Define $h = {\rm gcd}(n,a)$ and $k = {\rm gcd}(n,b)$, with $n = h \cdot n_h = k \cdot n_k$, $a = h \cdot a_h$, and $b = k \cdot b_k$.  Note that both ${\rm gcd}(n_h,a_h) = 1$ and ${\rm gcd}(n_k,b_k) = 1$.

    By this construction, the donut $(n_h,n_k,a_h,b_k)$ is primitive.  More, by Thm. 1 of \cite{2}, 

    \begin{equation} \label{nhnksum}
        n_hn_k = x + y + z,
    \end{equation}

    \noindent for some primitive Pythagorean triple $(x,y,z)$.  We claim that 

    \begin{equation}
        n = {\rm gcd}(h,k) \left[ x + y + z \right].
    \end{equation}

    Note that proving this claim consequently proves the result, as any constant multiple of a Pythagorean triple is itself a Pythagorean triple.  Moreover, if we are able to show that $n = {\rm lcm}(h,k)$, this claim will be proven, since Eqn. \ref{nhnksum} yields

    \begin{align}
         n^2 &= hk \left[ x + y + z\right]\\
         &= {\rm gcd}(h,k){\rm lcm}(h,k)\left[ x + y + z\right].
    \end{align}

     To that end, suppose $n$ has prime factorization given by 

     \begin{equation}
         n = 2^{e_0}p_1^{e_1} \ldots p_m^{e_m},
     \end{equation}

     \noindent written in this form since $n$ is necessarily even.  Because $n^2 = 2ab$, we know the prime factorizations of $a$ and $b$ take similar forms:

     \begin{equation}
         a = 2^{a_0}p_1^{a_1} \ldots p_m^{a_m},
     \end{equation}

    \noindent and

    \begin{equation}
         b = 2^{b_0}p_1^{b_1} \ldots p_m^{b_m},
     \end{equation}

     \noindent with the $a_i$ and $b_i$ exponents possibly being $0$.  Then, take $h = {\rm gcd}(n,a)$ and $k = {\rm gcd}(n,b)$.  In terms of prime factorizations, we have that 

     \begin{equation}
         h = 2^{{\rm min}(e_0, a_0)  }p_1^{ {\rm min}(e_1, a_1) } \ldots p_m^{ {\rm min}(e_m, a_m) } 
     \end{equation}

     \noindent and 

     \begin{equation}
         k = 2^{{\rm min}(e_0, b_0)  }p_1^{ {\rm min}(e_1, b_1) } \ldots p_m^{ {\rm min}(e_m, b_m) }. 
     \end{equation}

     We have that

     \begin{align}
         {\rm lcm}(h,k) = &2^{ {\rm max}( {\rm min}(e_0,a_0),{\rm min}(e_0,b_0)   ) }p_1^{ {\rm max}( {\rm min}(e_1,a_1),{\rm min}(e_1,b_1)   )  } \\ 
         &\ldots p_m^{ {\rm max}( {\rm min}(e_m,a_m),{\rm min}(e_m,b_m)   )  }. 
     \end{align}
    \end{proof}

    \noindent We also know that 

    \begin{equation} \label{exponent1}
        a_0 + b_0 + 1 = 2e_0,
    \end{equation}

    \noindent and

    \begin{equation}
        a_i + b_i = 2e_i
    \end{equation}

    \noindent for $i \geq 1$.  In either case it follows that either $a_i \leq e_i \leq b_i$ or $b_i \leq e_i \leq a_i$ (with the inequalities following in the case of Eqn. \ref{exponent1} since $a_0 \neq b_0$).  Hence, for all $i$, 

    \begin{equation}
        e_i = {\rm max}( {\rm min}(e_i,a_i), {\rm min}(e_i,b_i)), 
    \end{equation}

    \noindent showing that ${\rm lcm}(h,k) = n$, as desired.

\section{Conclusion}

The results in this paper are just the tip of the iceberg when it comes to investigating mathematical donuts.  A multitude of directions for potential research appear in \cite{2}, but it is worth noting some of the geometrically-aligned questions here.  

A very natural generalization of a mathematical donut (a two-dimensional object) would be to consider three-dimensional donuts.  Is there a classification of these donuts similar to that in \cite{2}?  How about classifying cube or cube-holed three-dimensional donuts?

Back in two-dimensions, there is no shortage of potential avenues to explore.  Can two-holed donuts be classified?  What about donuts that take the shape of an everyday waffle: an $n \times m$ array of equally spaced holes?  Both of these questions can be specified further by requiring the donuts to be \textit{perfect}, where the border between the hole and the exterior has a constant width.

Defining a cost function on donuts (for example, should the aforementioned donuts $(90,90,81,50)$ and $(90,90,75,54)$ cost the same, even though the holes have different dimensions?) or colorings on donuts (perhaps playing the role of icing on a donut) are two examples of connecting donuts to other mathematical areas.


\begin{thebibliography}{0}
\bibitem{1} H. Jacobs, \textit{Geometry: Seeing, Doing, Understanding}, 3rd ed., W. H. Freeman, New York, 2003.

\bibitem{2} W. Nirode and N. Krumpe, Donuts with Pythagoras, \textit{College Math. J.}  \textbf{53}(4) (2022) 306--311.

\bibitem{3} N. Sloane, The on-line encyclopedia of integer sequences (2020).


\end{thebibliography}
\end{document}